\documentclass[10pt]{amsart}
\usepackage{epsfig,url}
\usepackage{mathdots}
\usepackage{times}
\usepackage[T1]{fontenc}
\usepackage{mathrsfs}
\usepackage{latexsym}
\usepackage{graphicx}
\usepackage{epsfig}
\usepackage{hyperref, amsmath,amsfonts,amsthm,amssymb,amscd, cleveref}
\usepackage{xcolor}
\usepackage{colonequals}

\newtheorem{theorem}{Theorem}[section]
\newtheorem{lemma}[theorem]{Lemma}

\newtheorem{proposition}[theorem]{Proposition}

\theoremstyle{definition}
\newtheorem{definition}[theorem]{Definition}

\newcommand{\pMatrix}[4]{\left(\begin{matrix}#1 & #2 \\ #3 & #4\end{matrix}\right)}

\newcommand\be{\begin{equation}}
\newcommand\ee{\end{equation}}
\newcommand\bee{\begin{equation*}}
\newcommand\eee{\end{equation*}}

\newcommand{\sgn}{\operatorname{sgn}}

\newcommand{\ba}{\begin{eqnarray}}
\newcommand{\ea}{\end{eqnarray}}

\newcommand{\SL}{\operatorname{SL}}
\newcommand{\GL}{\operatorname{GL}}
\newcommand{\Q}{\mathbb{Q}}

\newcommand{\hol}{\operatorname{hol}}

\newcommand{\Hc}{\mathcal{H}}

\newtheorem{Definition}{\bf Definition}[section]

\newtheorem{Prop}[Definition]{\bf Proposition}
\numberwithin{equation}{section}

\theoremstyle{remark}
\newtheorem{remark}[theorem]{Remark}

\newenvironment{singnumalign}{
    \begin{equation}
    \begin{aligned}
}{
    \end{aligned}
    \end{equation}
    \ignorespacesafterend
}

\begin{document}

\title[Holomorphic]
{Holomorphic projection for sesquiharmonic Maass forms }
\author{Michael Allen}\address{Department of Mathematics, Louisiana State University, Baton Rouge, LA}
\email{allenm3@lsu.edu}
\author{Olivia Beckwith}\address{Department of Mathematics,
Tulane University, New Orleans, LA 70118}
\email{obeckwith@tulane.edu}
\author{Vaishavi Sharma}
\address{Mathematics Department, The Ohio State University, Columbus, OH 43210}
\email{sharma.1254@osu.edu}

%    General info
% \subjclass[2010]{Primary ??}

\date{\today}

% \keywords{??}

\begin{abstract}
We study the holomorphic projection of mixed mock modular forms involving sesquiharmonic Maass forms. As a special case, we numerically express the holomorphic projection of a function involving real quadratic class numbers multiplied by a certain theta function in terms of eta quotients. We also analyze certain shifted convolution $L$-series involving mock modular forms and bound certain shifted convolution sums.
\end{abstract}

\maketitle

\section{Introduction}
Harmonic Maass forms find widespread applications in the field of number theory and across mathematics, tracing back to Ramanujan's seminal discovery of mock theta functions and elucidated by the work of Zwegers \cite{Zwegers} and Bruinier and Funke \cite{BruinierFunke}. These functions have been employed in diverse contexts (see \cite{BringmannFolsomOnoRolen} for an overview), including partition statistics, quadratic number fields, representation theory, and various areas within mathematical physics. 

A prominent method for studying the coefficients of harmonic Maass forms is holomorphic projection, in which a real-analytic modular form is projected onto well-studied subspaces of holomorphic modular forms. Many authors have used this technique to show recurrence formulas for arithmetic functions including the coefficients of mock theta functions \cite{ImamogluRaumRichter}, the smallest parts function for partitions \cite{AhlgrenAndersen}, and class numbers \cite{MertensClassNumbers}. Holomorphic projection was also used in \cite{MertensOno} to show that certain shifted convolution Dirichlet series are coefficients of modular forms. More recently, holomorphic projection has been used to study divisibility properties of Hurwitz class numbers by the second author, Raum, and Richter (see \cite{BRR1}, \cite{BRR2}, \cite{BRR3}).  

All of the examples mentioned above compute the holomorphic projection of a product of a harmonic Maass form and a holomorphic modular form. A natural problem is to extend the scope of these methods to more general real-analytic modular forms. There are many possible generalizations one could consider by varying the type of modular form, the weight, and the type of product taken. For any choice, one would hope to use the integral formulas for holomorphic projection maps to compute formulas involving elementary functions and the coefficients of the original functions considered. This is difficult in the most general case due to the shifted convolution $L$-series appearing in the coefficients. 
One would then look for specific examples for which these formulas can be used to give new information about coefficients of modular forms. 

We consider the situation where the harmonic Maass form is replaced by a sesquiharmonic Maass form---a real-analytic modular form with shadow equal to a harmonic Maass form (see \Cref{subsec:sesqui}).  Examples of such forms have appeared recently with Fourier coefficients involving arithmetic constants such as real quadratic class numbers \cite{DukeImamogluToth2011-1} and non-critical $L$-values for holomorphic cusp forms \cite{BringmannDiamantisRaum}. Our main result, \Cref{thm:proj-computation}, computes the holomorphic projection of the product of a weight $1/2$ sesquiharmonic form and a weight $3/2$ theta function.  In Theorem \ref{thm:AAS-projection} we give the specialization of this formula to a sesquiharmonic Maass form defined in \cite{AhlgrenAndersenSamart} which has Fourier coefficients related to class numbers of both real and imaginary quadratic fields. 

To state Theorem \ref{thm:AAS-projection}, we need a few definitions. For any discriminant $d$, let $\mathcal{Q}_d$ be the set of binary quadratic forms of discriminant $d$ which are not negative definite, and let $\mathcal{Q}^{prim}_d$ be the subset of primitive forms in $\mathcal{Q}_d$. We let $h^+(d) \colonequals | \mathcal{Q}^{prim}_d / \SL_2(\mathbb{Z})|$ and note that when $d$ is a positive non-square this is equal to the narrow class number of the unique real quadratic order of discriminant $d$.

The Hurwitz class numbers count $\SL_2(\mathbb{Z})$ classes of binary quadratic forms inversely weighted by stabilizer size:
\begin{equation}\label{eq:Hurwitz-defn}
    H(n) \colonequals \!\!\!\!\!\!\!\!\!\! \sum_{Q \in \SL_2(\mathbb{Z}) \backslash \mathcal{Q}_{-n}} \frac{2}{\operatorname{Stab} (Q)},
\end{equation}
with the convention that $H(0) =  \frac{-1}{12}$ and $H(n)=0$ if $-n$ is neither zero nor a negative discriminant.  An analogue of $H(n)$ for positive non-square discriminants was defined in \cite{DukeImamogluToth2011-2} and for positive square discriminants in \cite{AhlgrenAndersenSamart} as follows.  For positive discriminants $d$, let $R(d)$ denote the regulator
\[
    R(d) \colonequals 
    \begin{cases}
        2 \log \epsilon_d & d \text{ non-square } \\
        2 \log {\sqrt{d}} & d \text{ square},
    \end{cases}
\]
where $\epsilon_d$ is the smallest unit $> 1$ of norm $1$ in the quadratic order $\mathcal{O}$ of discriminant $d$, and define the general Hurwitz function
\begin{equation}\label{eq:general-hurwitz-defn}
    h^*(d) \colonequals \frac{1}{2 \pi} \!\!\!\!\!\!\!\!\!\!\!\!\!\! \sum_{\substack{\ell^2 | d \\ d/\ell^2 \text{ is a discriminant}}}\!\!\!\!\!\!\!\!\!\!\!\!\!\! R(d/\ell^2) h^+(d/\ell^2).
\end{equation}

%%%%%%%%%%%%%%%%%%%%%%%%
%%%%%%%%%%%%%%%%%%%%%%%%%
%%%%%%%%%%%%%%%%%%%%%%%%%%%%
Let $\gamma$ denote the Euler--Mascheroni constant and define $\delta_{\square}(N)$ to be 1 if $N$ is a perfect square and $0$ otherwise. 

For any Dirichlet character $\chi$, we define 
\begin{align*}
r_{\chi}(h) &\colonequals  \delta_{\square}(h) \chi(\sqrt{h}) \sqrt{h} \left( \frac{\gamma - \log (16 \pi)}{4 \pi} + \frac{\gamma + \log (4 \pi \sqrt{h})}{4 \pi} +  \frac{1}{12 \sqrt{h}}\right) \\
&+\sqrt{\pi} h  \sum_{\substack{ n + m^2 = h \\ n <0 \\ m > 0}}  \frac{ H(|n|) m \chi(m)}{\sqrt{|n|} (m + \sqrt{|n|})m} + \sum_{\substack{n + m^2 = h \\ n, m > 0}} \frac{h^*(n)}{\sqrt{n}} \chi (m)m  \\
&+ \!\!\!\!   \sum_{\substack{m^2+n^2 =h \\ m, n > 0}} \!\! \frac{\chi (m)}{2\pi}  \left( 2 n \arctan\left( \frac{m}{n}\right) - m \log \left(\frac{4n^2}{h}\right) \right).
\end{align*}

We have the following special case of \Cref{thm:proj-computation}, which is obtained by setting $g = \theta_{\chi}$ and $F = Z$ (see Section \ref{sec:background}) in \Cref{thm:proj-computation}

\begin{theorem}\label{thm:AAS-projection}
Let $\chi$ be an odd Dirichlet character modulo $m$. The function $\sum_{h=1}^{\infty} r_{\chi}(h) q^h$ belongs to the space $S_2( \Gamma_0(4m^2))$. 
\end{theorem}

If we let $\chi_4(n) = \left( \frac{-4}{n} \right)$, Theorem \ref{thm:AAS-projection} tells us $\sum_{h=1}^{\infty} r_{\chi_4}(h) q^h$ belongs to the space $S_2( \Gamma_0(64))$. Using \cite{lmfdb}, we find that the space  $S_2( \Gamma_0(64))$ is spanned by $f_1, f_2, f_3,$ where $f_1$ is the newform
$$f_1 (z) \colonequals \frac{\eta(8z)^8}{\eta(4z)^2 \eta(16z)^2} = q +2q^5 - 3 q^9 - 6 q^{13} + 2 q^{17} + \cdots$$ and $f_2, f_3$ are the oldforms
$$f_2(z) \colonequals \eta(4z)^2 \eta(8z)^2 = q - 2 q^5 - 3 q^9 + 6 q^{13} + 2 q^{17} - q^{25} - \cdots $$ and 
$$f_3 (z) \colonequals f_2(z) | V(2) = q^2 - 2 q^{10} - 3 q^{18} + 6 q^{26} + 2 q^{34} - q^{50} - \cdots  $$ 

Explicitly, we have
\begin{align*}
\sum_{h=1}^{\infty} r_{\chi_4}(h) q^h &= \left(\frac{1}{2} r_{\chi_4}(1) + \frac{1}{4} r_{\chi_4}(5)\right) f_1 + \left(\frac{1}{2} r_{\chi_4}(1) - \frac{1}{4} r_{\chi_4}(5)\right) f_2 + r_{\chi_4}(2) f_3 \\
&\approx (0.028599) f_1 + (0.0000017) f_2 + (0.0579284) f_3 
\end{align*}

Various arithmetic patterns among the $r_{\chi_4}(h)$ follow---see Section \ref{sec:numerics} for further discussion. 

We can view the presence of the infinite sums in the coefficients as a weight $\frac{3}{2}$ Eisenstein version of a theorem of \cite{MertensOno}, connecting special values of shifted convolution $L$-series defined by \cite{HoffsteinHulse} to mock modular forms. 
Given series $g_i (\tau) = \sum_{n=1}^{\infty} a_i(n) q^n$, $i = 1,2$, Hoffstein and Hulse \cite{HoffsteinHulse} define the shifted convolution series 
$$
D(g_1,g_2, h, s)\colonequals \sum_{n = 1}^{\infty} \frac{a_1( n+h) \overline{a_2(n)} }{n^s}.
$$
If $g_1,g_2$ are holomorphic cusp forms of the same integral weight, then \cite{HoffsteinHulse} shows that $D(g_1,g_2,h,s)$ has a meromorphic continuation in $s$. Mertens and Ono \cite{MertensOno} related these series to modular forms and harmonic Maass forms by computing
\begin{equation}\label{eq:MertensOno}
\pi_{2}^{\hol, \operatorname{reg}} (M_{g_1}^- \cdot g_2) = (k-2)! \sum_{h =1}^{\infty} ( D( g_2, g_1, -h, k-1) - D(\overline{g_1}, \overline{g_2}, h, k-1)) q^h
\end{equation}
where $k$ is the weight of both $g_1$ and $g_2$, $\pi_{2}^{\hol, \operatorname{reg}}$ is a regularized version of the holomorphic projection operator, and $M_{g_1}^-$ is the nonholomorphic Eichler integral 
$$
M_{g_1}^- (\tau) = \sum_{n=1}^{\infty} n^{1 -k} \overline{a_1(n)} \Gamma(k-1, 4 \pi n y) q^{-n}.
$$
Formally, the relation \eqref{eq:MertensOno} holds for other weights and for non-cusp forms as well, even though the two shifted convolution sums on the right hand side may not converge. For $\operatorname{Re}(s) > \frac{3}{2}$, we have
\begin{equation}\label{eq:symmetrized}
D( \theta_{\chi}, H, -h, s) - D(H, \theta_{\overline{\chi}}, h, s)
= \sum_{m^2 - n = h}^{\infty} H(n) m \chi(m) \left( \frac{1}{n^s} - \frac{1}{m^{2s}}    \right) 
\end{equation}
 When $s = \frac{1}{2}$, the series on the left do not converge but the series on the right does and is equal to a sum appearing in Theorem \ref{thm:AAS-projection}:
$$ 
 \sum_{\substack{m^2 + n = h \\ n < 0 \\ m > 0 }}^{\infty} H(|n|) m \chi(m) \left( \frac{1}{\sqrt{|n|}} - \frac{1}{m} \right) =
h \sum_{\substack{ n + m^2 = h \\ n <0 \\ m > 0}} \! \frac{ H(|n|) m \chi(m)}{ \sqrt{|n|} (m + \sqrt{|n|})m}
$$
 In Section \ref{sec:shiftedconvolution}, we analyze the second series on the left hand side of \eqref{eq:symmetrized} and prove bounds for related shifted convolution sums. These bounds are proved using methods from a recent preprint by Walker \cite{Walker}, who used the spectral theory of automorphic forms to study the self-correlations of Hurwitz class numbers.

In Section \ref{sec:background}, we introduce the needed background on modular forms, sesquiharmonic Maass forms, and holomorphic projection. In Section \ref{sec:shiftedconvolution}, we analyze the poles of certain shifted convolution $L$-series and use them to prove a bound for shifted convolution sums of mock modular coefficients and cusp form coefficients. In Section \ref{sec:proofs} we prove our main theorem and Theorem \ref{thm:AAS-projection}. Section \ref{sec:numerics} provides some numerical data and patterns related to the $r_{\chi_4}(n)$ that follow from Theorem \ref{thm:AAS-projection}. 

\section{Background}\label{sec:background}
\subsection{Modular forms}
Here we give some standard notation and facts from the theory of modular forms. See a text such as Chapters 1-3 of \cite{OnoCBMS} for details.
Throughout, $\mathbb{H} \colonequals \{ z \in \mathbb{C}: \operatorname{Im}(z) > 0 \}$ denotes the upper half plane and $z = x+iy$ represents an element of $\mathbb{H}$, with $x,y \in \mathbb{R}$. We let $e(w) = e^{2 \pi iw}$ for all $w \in \mathbb{C}$ and $q^n \colonequals e(n z)$ for any $n \in \mathbb{Q}$. 
We will use the action of $ \GL_2^+(\Q)$ on $\mathbb{H}$ given by fractional linear transformation:
$$
\pMatrix abcd   \cdot z \colonequals \frac{az+b}{cz+d}
$$
and for any function $f : \mathbb{H} \to \mathbb{C}$ and $k \in \frac{1}{2} \mathbb{Z}$, we let 
$$
f|_k \pMatrix abcd  (z) \colonequals (cz+d)^{-k} f \left( \pMatrix abcd   \cdot z \right).
$$ 

For $k, N \in \mathbb{Z}$, $N \ge 1$, and any Dirichlet character $\chi$ modulo $N$, let $M_k(N,\chi)$ and $S_k(N,\chi)$ denote the usual spaces of weight $k$ holomorphic modular forms and cusp forms on $\Gamma_0(N)$ with character $\chi$, simplified as $M_k(N), S_k(N)$ when $\chi$ is trivial.

The $\theta$ multiplier system is given by 
$$
\nu_{\theta} (\gamma ) \colonequals \epsilon_d^{2k} \left( \frac{c}{d} \right)
$$
for $\gamma = \pMatrix abcd \in \Gamma_0(4)$, where
$$
\epsilon_d = \begin{cases}
1 & d \equiv 1 \pmod{4} \\
i & d \equiv 3 \pmod{4},
\end{cases}
$$
and $\left( \frac{c}{d} \right)$ is the Kronecker symbol.

For $k \in \frac{1}{2} + \mathbb{Z}$, $N \in \mathbb{N}$, and any Dirichlet character $\chi$ modulo $4N$, we let $M_k(4N,\chi)$ (resp. $S_k(4N,\chi)$) be the space of holomorphic functions $f:\mathbb{H} \to \mathbb{C}$ which satisfy 
$$
f\left( \frac{az+b}{cz+d} \right) = \chi(d) \nu_{\theta}^{2k}( \gamma )  (cz+d)^k f(z) 
$$
for all $\gamma = \pMatrix abcd \in \Gamma_0(4N)$ and which are holomorphic (resp. vanishing) at all cusps of $\Gamma_0(4N)$.

\subsection{Holomorphic unary theta functions}
Fundamental examples of half integral weight modular forms are given by theta functions. For any Dirichlet character $\chi$ modulo $m$, let
\begin{equation}\label{eq:theta-defn}
    \theta_{\chi}(z) \colonequals \sum_{ n \in \mathbb{Z}} \chi (n) n^{\nu} q^{n^2}
\end{equation}
where $\nu \colonequals \frac{1- \chi(-1)}{2}$.
Then we have (see Theorem 1.44 \cite{OnoCBMS})
$$
\theta_{\chi} \in \begin{cases} M_{\frac{1}{2}} (4 m^2, \chi) & \nu =0 \\
S_{\frac{3}{2}} (4 m^2, \chi \chi_4) & \nu = 1.
\end{cases}
$$

%%%
\subsection{Holomorphic Projection}
For any congruence subgroup $\Gamma$ and integer $k$, we let $\mathbb{S}_k(\Gamma)$ denote the space of real-analytic functions $f :\mathbb{H} \to \mathbb{C}$ which transform as weight $k$ modular forms with respect to $\Gamma$ and have exponential decay at cusps.

Sturm \cite{SturmHolProj} proved a formula for the holomorphic projection of a modular form with certain growth conditions and weight greater than 2. This formula was modified by \cite{GrossZagier} to include weight $2$ and allow for more general growth conditions at cusps. 

\begin{theorem}[Proposition 6.2 \cite{GrossZagier}]\label{thm:proj-formula}
Let $k \ge 2$ be an integer and $f (z)= \sum_{n \in \mathbb{Z}} a(n,y) q^n \in \mathbb{S}_k (\Gamma_1(N))$.
Let
$$
a_n \colonequals \lim_{s \to 0} \frac{ (4 \pi n)^{k-1}}{(k-2)!} \int_0^{\infty} a(n,y) e^{- 4 \pi n y} y^{s+k-2} dy.
$$
Then
$$
\pi_k^{hol} (f) (z)  \colonequals \sum_{n =1}^{\infty} a_n q^n
$$
lies in $S_k(\Gamma_1(N))$ and satisfies the property that for all $g \in S_k(\Gamma_1(N))$, we have
\begin{equation}\label{eq:projection-condition}
\langle f, g \rangle_{\mathrm{Pet}} = \langle \pi_k^{hol}(f),g \rangle_{\mathrm{Pet}}.
\end{equation}
\end{theorem}

This is proven in \cite{SturmHolProj} for $k >2$. For $k =2$, take $\alpha(M) = \beta(M)=0$ in Proposition 6.2 of \cite{GrossZagier}, and replace $\Gamma_0(N)$ with $\Gamma_1(N)$. The $k=2$ case is also used in \cite{BRR1} for $\Gamma_1(N)$. In that formulation the holomorphic projection can be quasimodular, which isn't possible here because we are assuming $f$ has exponential decay at cusps.  

\subsection{Harmonic Maass forms}\label{subsec:harmonic}

To define the classes of real-analytic modular forms of interest, we need two differential operators.  We have the Bruinier-Funke  operator $\xi_k$ defined by 
$$
\xi_k = i y^k \overline{ \left( \frac{\partial}{\partial x} + i \frac{ \partial}{\partial y} \right) }.
$$
and the weight $k$ hyperbolic Laplacian
$$
\Delta_k = -y^2 \left( \frac{ \partial^2}{\partial x^2} + \frac{\partial^2}{ \partial y^2} \right) + iky \left( \frac{ \partial}{\partial x} + i \frac{\partial}{\partial y} \right) = \xi_{2-k} \xi_k.
$$

One can verify the essential relation 
$$
\xi_k (F) |_{2-k} \gamma = \xi_k ( F|_k \gamma),
$$
which shows that $\xi_k$ maps functions on $\mathbb{H}$ which transforms as weight $k$ modular forms to functions which transform as weight $2-k$ modular forms.

Let $k \in \frac{1}{2} \mathbb{Z}$, and let $N$ be a positive integer such that $4|N$ if $k \not\in \mathbb{Z}$.

\begin{definition}
We call a real-analytic function $f: \mathbb{H} \to \mathbb{C}$ a harmonic Maass form of weight $k$ on $\Gamma_0(N)$ if
\begin{enumerate}
\item For all $\gamma \in \Gamma_0(N)$, we have
$f|_k \gamma = \begin{cases}
f & k \in \mathbb{Z} \\
\nu_{\theta}(\gamma)^{2k} f    & k \in \frac{1}{2}+\mathbb{Z}.
\end{cases}$
\item We have $\Delta_k f = 0$.
\item The function $f$ has polynomial growth at all the cusps of $\Gamma_0(N)$.
\end{enumerate}
\end{definition}
We let $H_k(N)$ denote the space of such functions.

It is well known (see \cite{BringmannFolsomOnoRolen}) that each $f \in H_k(N)$  has a Fourier expansion of the form
\begin{equation}\label{eq:HarmonicFourier}
f(z) = \sum_{n \gg -\infty} c^+(n) q^n + c^-(0) y^{1-k} + \sum_{\substack{n \ll \infty \\ n \neq 0 }} c^-(n) \Gamma(1-k, 4 \pi |n| y ) q^{-n}.
\end{equation}

In $1975$, Zagier \cite{Zagier1} discovered the weight $\frac{3}{2}$ moderate growth harmonic Maass form on $\Gamma_0(4)$, given by 
\begin{align} \label{eq:Hcdef}
\Hc(\tau)  := -\frac{1}{12} +  \sum_{n \geq 1} H(n) q^n + \frac{1}{8\pi\sqrt{v}} + \frac{1}{4\sqrt{\pi}} \sum_{n \geq 1} n  \Gamma\left(-\frac{1}{2},4\pi n^2 v \right) q^{-n^2}.
\end{align}
Theorem 1.2 of \cite{BeckwithMono} defines higher level versions of this function. 

\subsection{Sesquiharmonic Maass forms}\label{subsec:sesqui}
 
\begin{definition}
A real-analytic function $F: \mathbb{H} \to \mathbb{C}$ is a sesquiharmonic Maass form of weight $k$ with respect to $\Gamma_0(N)$ if 
\begin{enumerate}
\item For all $\gamma \in \Gamma_0(N)$, we have $F|_k \gamma = \begin{cases}
F & k \in \mathbb{Z} \\
\nu_{\theta}(\gamma)^{2k} F    & k \in \frac{1}{2}+\mathbb{Z}.
\end{cases}$
\item $(\xi_k \circ \Delta_k) F = 0$.
\item $F$ has at most linear exponential growth at cusps.
\end{enumerate}
\end{definition}
Let $V_k (N) $ be the complex vector space of such functions $F$ for which $\Delta_k(F)$ lies in $M_k(N)$.

To describe the Fourier expansions of these functions when $k = \frac{1}{2}$ we'll use the special functions $\alpha(y)$ and $\beta(y)$ defined in \cite{DukeImamogluToth2011-1} as
\begin{equation}\label{eq:alphadef}
\alpha(y) \colonequals \frac{\sqrt{y}}{4 \pi} \int_0^{\infty} t^{-1/2} e^{- \pi y t} \log ( 1 + t) dt, \qquad y > 0,
\end{equation}
and
\begin{equation}\label{eq:betadef}
\beta (y) \colonequals \frac{1}{\sqrt{\pi}} \Gamma \left(\frac{1}{2},\pi y \right) = \frac{1}{\sqrt{\pi}} \int_{\pi y}^{\infty} t^{1/2 -1} e^{-t} dt.
\end{equation}
%%%%%%

\begin{proposition}\label{prop:fourier-expansion}
Let $F \in V_{\frac{1}{2}} (4N)$. Then $F$ has a Fourier expansion of the form
\begin{singnumalign}\label{eq:SesquiFourier}
F(\tau) &=\left( d(0) + d(1) \log y + d(2) y^{1/2} + d(3) y^{1/2} \log y \right) + \sum_{\substack{n \gg -\infty \\ n \ne 0}} c(n)q^n \\
&+ \sum_{\substack{n \ll \infty \\ n \ne 0}} b(n) \beta(-4 n y) q^n + \sum_{n \ge 1} a(n) \alpha(4 n y) q^{n}.
\end{singnumalign}
\end{proposition}
\begin{remark}
We refer to the portions of this decomposition involving the $d(n)$, $c(n)$, $b(n)$, and $a(n)$ coefficients as the constant term, the holomorphic term, the harmonic term, and the sesquiharmonic term, respectively.
\end{remark}
\begin{proof}
We use Theorem 3.3 of \cite{ALR}, which is stated for even integer weights but also holds for half integral weights (for example, it is used in \cite{Matsusaka2}). Let $W_{\kappa,\mu} (y)$ be the W-Whittaker function, and let $\mathcal{M}_{\kappa,\mu}(y)\colonequals W_{-\kappa, \mu} (ye^{\pi i})$. If a real-analytic periodic function $F(\tau)$ on $\mathbb{H}$ satisfies $\Delta_k^2 F(\tau) = 0$, where $k \neq 1$, then Theorem 3.3 of \cite{ALR} implies that $F$ has a Fourier expansion of the form
\begin{equation}\label{eq:biharmonic-fourier}
F(z) = \sum_{n \in \mathbb{Z}} e^{2 \pi i n x} \left( \sum_{j=0}^1 a_{n,j}^- u_{k, n}^{j,-} (y) + a_{n,j}^+ u_{k,n}^{j,+} (y) \right)
\end{equation}
where for $n \neq 0$ we have
\begin{equation*}
u_{k,n}^{j,-}(y) \colonequals y^{-k/2} \frac{\partial^j}{\partial s^j} W_{\sgn(n) \frac{k}{2}, s + \frac{k-1}{2}} ( 4 \pi | n | y) \Big|_{s=0}
\end{equation*}
and
$$
u_{k,n}^{j,+}(y) \colonequals y^{-k/2} \frac{\partial^j}{\partial s^j} \mathcal{M}_{\sgn(n) \frac{k}{2}, s + \frac{k-1}{2}} ( 4 \pi | n | y) \Big|_{s=0},
$$
and for $n=0$ we have
$$
u_{k,0}^{j,+}(y)  \colonequals  (\log y)^j
$$
and
$$
u_{k,0}^{j,-}(y) \colonequals  (-1)^j (\log y)^j y^{1-k}.
$$
We will show that when $k = \frac{1}{2}$, each term in \eqref{eq:biharmonic-fourier} can be expressed via terms in \eqref{eq:SesquiFourier} for $F \in V_{\frac{1}{2}} (\Gamma_0(4N))$.
By (2.2) of \cite{Matsusaka2}, we have
$$
u_{k,n}^{0,-} (y) e^{2 \pi i n x} = 
\begin{cases}(4 \pi n)^{k/2} q^n & n > 0 \\
 (4 \pi |n| )^{k/2} \Gamma(1-k,4 \pi |n| y) q^n  & n < 0.
\end{cases}
$$
Thus the $u_{\frac{1}{2},n}^{0,-}(y) e^{2 \pi i n x}$ terms in \eqref{eq:biharmonic-fourier} are scalar multiples of $q^n$ for $n>0$ and are scalar multiples of $\beta(-4ny)q^n$ for $n<0$, so these terms are part of the holomorphic and harmonic terms in \eqref{eq:SesquiFourier}, respectively. 

Similarly by (2.3) of \cite{Matsusaka2}, we have for $n<0$
$$
u_{k,n}^{0,+} (y) e^{2 \pi i n x} = (4 \pi |n| e^{\pi i})^{\frac{k}{2}} q^n.
$$
It follows that the $u_{\frac{1}{2},n}^{0,+}(y) e^{2 \pi i n x}$ terms in \eqref{eq:biharmonic-fourier} are multiples of $q^n$ for $n<0$. Since $F$ has linear exponential growth as $y \to \infty$, there are only finitely many of these terms. 

For the other terms in \eqref{eq:biharmonic-fourier}, we use Lemma 2.2 of \cite{Matsusaka1}, which says that if we set $u_{k,n}^{j,-} (y) = 0$ if $j <0$, then we have the relations
\begin{equation}\label{eq:shadow1}
\xi_k \left( u_{k,n}^{j,-} (y) e^{2 \pi i n x}\right) = 
\begin{cases} (j(1-k) u_{2-k, -n}^{j-1,-}(y) + (1-k) u_{2-k, -n}^{j-2,-}(y))e^{-2 \pi i n x} & n > 0 \\
 - u_{2-k,- n}^{j,-} (y) e^{- 2 \pi i n x}  & n < 0.
\end{cases}
\end{equation}
and
\begin{equation}\label{eq:shadow2}
\xi_k \left( u_{k,n}^{j,+} (y) e^{2 \pi i n x} \right) = 
\begin{cases} - u_{2-k,- n}^{j,+} (y) e^{- 2 \pi i n x} & n > 0 \\
(j(1-k) u_{2-k, -n}^{j-1,+} + (1-k) u_{2-k, -n}^{j-2,+}(y))e^{- 2 \pi i nx} & n < 0.
\end{cases}
\end{equation}
We deduce from these relations and the assumption that $\xi_{\frac{1}{2}} \xi_{\frac{3}{2}} \xi_{\frac{1}{2}} F = 0$, that $a_{1,n}^- =0$ if $n<0$ and $a_{1,n}^+ =0$ if $n >0$. Similarly, for $n<0$, we have that $\Delta_{\frac{1}{2}} (u_{\frac{1}{2},n}^{1,+} (y) e^{2 \pi i n x}) $ is a scalar multiple of $q^n$, so by our assumption that $\Delta_{\frac{1}{2}} F \in M_{\frac{1}{2}} (4N)$, we have $a_{1,n}^+ = 0$. 

For $n>0$, it follows from \eqref{eq:shadow1} that $\xi_{\frac{1}{2}}(u_{\frac{1}{2},n}^{0,+}(y) e^{2 \pi i n x})$ is a scalar multiple of $q^{-n}$. It is straightforward to check that $\xi_{1/2} ( \beta(-4 n y) q^{n})$ is also a multiple of $q^{-n}$, and it follows that for $n>0$, $u_{\frac{1}{2},n}^{0,+}(y) e^{2 \pi i n x}$ can be expressed as a linear combination of a holomorphic function and  $\beta(-4 n y) q^n$. 

Using 8.853(3) of \cite{GradshteynRyzhik}, one can show that for $n >0$,
$$
8 \pi \alpha(4 n y) = \int_{4 n \pi y}^{\infty} \Gamma \left(-\frac{1}{2},t \right) e^t t^{-1/2} dt. 
$$
It follows that for $n > 0$,
$$
\xi_{\frac{1}{2}} ( \alpha(4 n y) q^n) = \frac{1}{8 \pi} q^{-n} \Gamma \left(-\frac{1}{2},4 \pi n y \right) (4 \pi n)^{\frac{1}{2}}.
$$
Using \eqref{eq:shadow1}, we deduce that for $n>0$, $u_{\frac{1}{2},n}^{1,-} (y) e^{2\pi i n x}$ is a linear combination of $ \alpha(4 n y) q^n$ and holomorphic terms, which can also be shown explicitly by differentiating the integral formula for $W_{\frac{1}{4},  \frac{1}{4}-s} (4 \pi |n| y)$ given by 9.222 of \cite{GradshteynRyzhik}.

Thus the Fourier expansion in \eqref{eq:biharmonic-fourier} takes the form of \eqref{eq:SesquiFourier} for $F \in V_{\frac{1}{2}} (4N)$, completing the proof.
\end{proof}

A sesquiharmonic Maass form whose coefficients involve class numbers of positive discriminant was discovered in \cite{DukeImamogluToth2011-1}. The coefficients of square index of this function were computed in \cite{AhlgrenAndersenSamart}. Here we briefly state Theorem 2 of \cite{AhlgrenAndersenSamart}.
\begin{theorem}[Theorem 2 \cite{AhlgrenAndersenSamart}]\label{thm:AAS-form}
With $H(d), h^*(d), \alpha(n)$, and $\beta(n)$ defined as in \eqref{eq:Hurwitz-defn}, \eqref{eq:general-hurwitz-defn}, \eqref{eq:alpha-defn}, and \eqref{eq:betadef}, respectively, the function $Z(\tau)$ with Fourier expansion given by
\begin{align*}
    Z( \tau) = &\sum_{d > 0} \frac{h^* (d)}{\sqrt{d}} q^d + \sum_{d < 0} \frac{H(|d|)}{\sqrt{|d|}} \beta(4 |d| y) q^d  \\
    &+ \sum_{n \geq 1} 2\alpha(4 n^2 y) q^{n^2} + \left( \frac{\gamma - \log(16 \pi)}{ 4 \pi} \right) - \frac{ \log y}{4 \pi} + \frac{\sqrt{y}}{3}
\end{align*}
is a sesquiharmonic Maass form of weight $\frac{1}{2}$ for $\Gamma_0(4)$.  Recall that here $\gamma$ is the Euler--Mascheroni constant. 
\end{theorem} 
 This function is a shift of the function $Z_+ (\tau)$ in \cite{DukeImamogluToth2011-1} by a constant multiple of $\theta(\tau)$ and will be used in Theorem \ref{thm:AAS-projection}. Higher level analogs of $Z(\tau)$ were defined in \cite{BeckwithMono}. In particular, for every square free odd $N$, Theorem 1.3 of that paper gives the Fourier expansion of the analogue of $Z(\tau)$ for the group $\Gamma_0(4N)$.

%%%%%%%%%
\section{Shifted convolution sums}\label{sec:shiftedconvolution}
Throughout this section, let $f \in H_{\frac{3}{2}} (4N)$ have a Fourier expansion as in \eqref{eq:HarmonicFourier} and assume that $f$ has at most polynomial growth at all cusps. Let $g = \sum_{m=1}^{\infty} \ell(m) q^{m^2} \in S_{\frac{3}{2}} (4N, \chi)$ be a cusp form with coefficients supported on square indices. We furthermore assume that $f$ and $g$ satisfy the properties that $f \cdot g $ vanishes at cusps, and for any $h >0$, for all $\epsilon >0$ we have $\overline{c^+(m^2 - h)} \ell(m) = O(m^{2 + \epsilon})$ as $m \to \infty$. 

Define 
$$D_h(s) := \sum_{m^2 >h}^{\infty} \frac{\overline{c^+(m^2-h)} \ell(m)}{m^{2s+ 1}}
$$
for $\operatorname{Re}(s) > 1$. The growth condition above implies that the series converges absolutely on this region. 

Throughout the section, let $\theta \le \frac{7}{64}$ represent partial progress towards the Ramanujan-Petersson Conjecture for weight 0 Hecke-Maass cusp forms. 

\begin{Prop}\label{thm:ShiftedDirichletSeries}
For any $h > 0$, the function $D_h(s)$ has a meromorphic continuation to all of $s \in \mathbb{C}$ which is analytic on the half-plane $Re(s) > \frac{1}{2} + \theta.$
\end{Prop}
\begin{proof}

Let $P_h(\tau,s)$ be the weight 0 level $N$ Maass-Poincare series defined by
$$
P_h(\tau,s) := \sum_{\gamma \in \Gamma_{\infty} \backslash \Gamma_0(N)} \operatorname{Im}(\gamma \tau)^s e\left( h \gamma \tau \right) 
$$
for $Re(s) > 2$ and by meromorphic continuation to $s \in \mathbb{C}$. We put $f(\tau) = \sum_{n \in \mathbb{Z}} c(n,y) e^{2 \pi i n x}$. Following the method and much of the notation of \cite{Walker}, Rankin--Selberg unfolding gives
\begin{align*}
I_h(s) &\colonequals < y^{\frac{3}{2}} g(\tau) \overline{f}(\tau) , P_h (\tau, \overline{s} ) >  \\
&= \int_0^{\infty} \int_0^1 y^{\frac{3}{2} + s - 2} g(\tau) \overline{f}(\tau) \overline{q^{h}} dx dy \\
&= \sum_{m=1}^{\infty} \int_0^{\infty} y^{s + \frac{3}{2} -2} \ell(m) e^{- 2 \pi m^2 y}  \overline{c(m^2-h, y)} e^{- 2 \pi h y} dy \\
&=  \sum_{m > \sqrt{h}} ( \cdots)  +\sum_{m = \sqrt{h}} ( \cdots)  +\sum_{0 < m < \sqrt{h}} ( \cdots)  \\
&\equalscolon I_h^+(s) + I_h^0(s) + I_h^- (s).
\end{align*}
Here and throughout, $<,>$ denotes the Petersson inner product. 
The discussion on page 18 of \cite{Walker} shows that $I_h(s)$ has a meromorphic continuation to $s \in \mathbb{C}$, which is analytic for $\operatorname{Re}(s) > \frac{1}{2} + \theta$. 

For the holomorphic part, we have:
\begin{align*}
 I_h^+(s) &= \sum_{m > \sqrt{h}} \int_0^{\infty} y^{s - \frac{1}{2}} \ell(m) e^{- 2 \pi m^2 y}  \overline{c(m^2-h, y)} e^{- 2 \pi h y} dy \\ 
 &= \sum_{m > \sqrt{h}} \ell(m)   \overline{c^+(m^2-h)} \int_0^{\infty} y^{s -\frac{1}{2}}  e^{- 2 \pi (m^2-h) y} e^{- 2 \pi m^2 y} e^{- 2 \pi h y} dy \\ 
 &=  (4 \pi)^{-(s +\frac{1}{2})} \Gamma\left(s + \frac{1}{2}\right) D_h(s) \\ 
\end{align*}

For the constant term, we compute
\begin{align*}
I_h^0(s) &=  \delta_{\square}(h) \int_0^{\infty} y^{s - \frac{1}{2}} \ell(\sqrt{h}) e^{- 2 \pi h y} \overline{ c(0, y)} e^{- 2 \pi h y} dy \\ 
&=  \delta_{\square}(h) \ell(\sqrt{h}) \overline{c^+(0)} (4 \pi h)^{-(s + \frac{1}{2})} \Gamma \left(s + \frac{1}{2} \right) +  \delta_{\square}(h) \ell(\sqrt{h}) \overline{c^-(0)} (4 \pi h)^{-s} \Gamma(s)   \\ 
\end{align*}
Thus, $I_h^0(s)$ is a meromorphic function on $\mathbb{C}$ with poles contained within $\frac{1}{2} \mathbb{Z}_{\le 0}$.

For the nonholomorphic part, we compute (using \cite{GradshteynRyzhik} 6.455 in the last step)
\begin{align*}
 I_h^- (s) &=  \!\!\!\!\sum_{0<m < \sqrt{h}}  \int_0^{\infty} y^{s +\frac{3}{2} - 2} \ell(m) \overline{c^- (m^2-h)} \Gamma(1-3/2, |4 \pi (m^2-h) y| ) e^{-4 \pi m^2y}dy \\
    &=  \sum_{0<m < \sqrt{h}}\ell(m) \overline{c^- (m^2-h)} \int_0^{\infty} y^{s -\frac{1}{2}}  e^{- 4 \pi m^2 y}   \Gamma(-1/2, |4 \pi (m^2-h) y| )  dy \\
    &=  \sum_{0<m < \sqrt{h}} \ell(m) \overline{c^- (m^2-h)}  \frac{ |4 \pi (m^2-h)|^{-\frac{1}{2}} \Gamma(s)}{(s + \frac{1}{2}) (4 \pi \ell)^s} {}_2 F_1 \left(1, s, s+1; \frac{4 \pi m^2}{4 \pi h} \right).  \\
 \end{align*}
Thus, $I_h^-(s)$ is a meromorphic function on $\mathbb{C}$ with poles contained within $\{- \frac{1}{2} \} \cup \mathbb{Z}_{\le 0}$.

We deduce from the analyticity of $I_h(s)$ on $\operatorname{Re}(s) > \frac{1}{2}+\theta$ that $D_h(s)$ is also analytic on this half-plane.
\end{proof}

The next result is a modification of Theorem 9.1 of \cite{Walker}. 

\begin{Prop}\label{thm:Theorem9.1}
Fix $0< \epsilon < \theta$ and $h > 0$. In the vertical strip $\operatorname{Re}(s) \in \left(\frac{1}{2} + \epsilon, \frac{3}{2} + \epsilon\right)$ away from poles of $D_h(s)$, we have
$$
D_h(s) \ll_{\epsilon} |s|^{\epsilon + \frac{5}{2}\left(\frac{3}{2} - \operatorname{Re}(s)\right)}
$$
\end{Prop}
\begin{proof}
In the proof in \cite{Walker} of Theorem 9.1, we replace $F(z) = y^{\frac{3}{2}} | H(\tau)|^2$ with $F(\tau) = y^{\frac{3}{2}} g(\tau) \overline{f}(\tau)$, and the proof requires no adjustments. 
\end{proof}

\begin{Prop}\label{thm:ShiftedConvBound}
With the notation and assumptions on $f_1, f_2$ given at the beginning of the section, for any integer $h > 0$, there exists $\delta>0$ such that
$$
\sum_{m \ll X} \overline{c^+(m^2 -h)} \ell(m)
    \ll_{h} X^{\frac{3}{2} - \delta}.
$$
as $X \to \infty$. 
\end{Prop}

\begin{remark}
In our proof, $\delta$ depends on $\theta$, as well as $f,g, h$.
If we set $\theta =\frac{7}{64}$, we are able to set $\delta < \frac{57}{349} = 0.1633 \cdots$ According to the Ramanujan-Petersson Conjecture, we can set $\theta = 0$, producing $\delta < \frac{1}{6}$ and
$$
\sum_{m \ll X} \overline{c^+(m^2 -h)} \ell(m) \ll X^{\frac{4}{3} + \epsilon}
$$
This is supported by numerical evidence---see \Cref{fig:shifted-convergence} below for data in the case $f = \mathcal{H}(\tau)$ and $g(\tau) = \theta_{\chi_4}$. It would be interesting to know the optimal bound of the form $X^{c + \epsilon}$. Numerically, it appears that $c$ lies in the interval $[1, \frac{5}{4})$.
\end{remark}

\begin{figure}[!ht]
\begin{center}
    \includegraphics[scale=0.45]{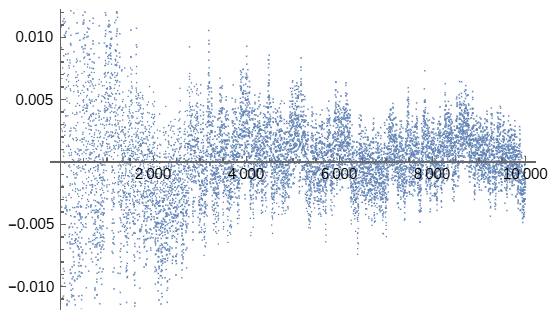}
\end{center}
\caption{Values of $\sum_{m \ll X} H(m^2-14) m\chi_4(m)/X^{5/4}$, where each $x$-value in the plot corresponds to $X = (2x+1)^2$.  Note that this suggests a stronger bound than the $4/3+\epsilon$ given above.}
\label{fig:shifted-convergence}
\end{figure}

\begin{proof}

Applying the truncated Perron's formula, for $\epsilon>0$ we find
\begin{equation}\label{eq:Perron}
  \left| \sum_{m \ll X} c^+(m^2 -h) \ell(m)\right|  = \frac{1}{2 \pi i } \int_{\frac{3}{2} + \epsilon - iT}^{\frac{3}{2} + \epsilon + iT} D_{\ell} \left(s - \left(
  \frac{3}{2}-1 \right) \right) \frac{X^s}{s} ds + O \left( \frac{X^{\frac{3}{2} + \epsilon}}{\sqrt{T}} \right).
\end{equation}
 For $a > \frac{1}{2} + \theta$, the Residue Theorem gives us
 \begin{equation}\label{eq:residue}
  \int_{\frac{3}{2} + \epsilon - iT}^{\frac{3}{2} + \epsilon + iT} D_{\ell} \left(s -
  \frac{1}{2} \right) \frac{X^s}{s} ds  =  \left( \int_{a - iT}^{a + iT} +  \int_{a +  iT}^{\frac{3}{2} + \epsilon + iT} - \int_{a - iT}^{\frac{3}{2} + \epsilon - iT} \right) D_{\ell} \left(s -\frac{1}{2} \right) \frac{X^s}{s} ds 
 \end{equation}
  
We bound the integral along the vertical contour using Proposition \ref{thm:Theorem9.1}:
\begin{equation}\label{eq:vertestimate}
2 \int_0^T t^{\epsilon} t^{\frac{5}{2}\cdot (\frac{3}{2} - a)} \frac{X^a}{(a + t)} dt \ll X^a T^{\epsilon + \frac{5}{2}(\frac{3}{2} - a)}
\end{equation}

We use Proposition \ref{thm:Theorem9.1} for the horizontal contours as well:
\begin{align}\label{eq:horestimate}
\int_a^{\frac{3}{2}+ \epsilon} |r + iT|^{\epsilon} |r + iT|^{\frac{5}{2} ( \frac{3}{2} - r)} \frac{X^r}{|r + iT|} dr &\ll  T^{\frac{15}{4}-1 + \epsilon} \int_a^{\frac{3}{2}+ \epsilon}  \left(\frac{X}{T^{\frac{5}{2}}} \right)^r dr \nonumber \\ 
&\ll  T^{\frac{15}{4}-1 + \epsilon} \frac{1}{\log \left(\frac{X}{T^{\frac{5}{2}}} \right) } \left(\frac{X}{T^{\frac{5}{2}}} \right)^r \biggr\vert_{r=a}^{r=\frac{3}{2} + \epsilon} \nonumber \\ 
&\ll \frac{1}{\log \left(\frac{X}{T^{\frac{5}{2}}} \right)}  \begin{cases} 
X ^{ \frac{3}{2} + \epsilon} T^{- 1 - \frac{3}{2} \epsilon} &  T^{\frac{5}{2}} \ll X \\
X^a T^{\frac{1}{4} a + \epsilon} &  T^{\frac{5}{2}} \gg X \\
\end{cases}
\end{align}
Let $T = X^b$. Then \eqref{eq:Perron}, \eqref{eq:residue},\eqref{eq:vertestimate}, and \eqref{eq:horestimate} give the result:
\begin{align}\label{eq:Prop3.3Final}
&\left| \sum_{m \ll X} c^+(m^2 -h) \ell(m) \right|  = \frac{1}{2 \pi i } \int_{\frac{3}{2} + \epsilon - iT}^{\frac{3}{2} + \epsilon + iT} D_{\ell} \left(s - \left(
  \frac{3}{2}-1 \right) \right) \frac{X^s}{s} ds + O \left( \frac{X^{\frac{3}{2} + \epsilon}}{\sqrt{T}} \right) \nonumber \\
    &\ll  X^{a + b( \epsilon + \frac{5}{2}(\frac{3}{2} - a))} + X^{ \frac{3}{2} + \epsilon - b(1 + \frac{3}{2} \epsilon)} +O \left(X^{\frac{3}{2} + \epsilon - \frac{b}{2}} \right).
  \end{align}
  If $\epsilon$ is sufficiently small, we can choose $b$ such that
  $$
2\epsilon < b < \frac{\frac{3}{2} - a}{\epsilon + \frac{5}{2}(\frac{3}{2}-a)}  
  $$
  and all three exponents in \eqref{eq:Prop3.3Final}
  are strictly smaller than $\frac{3}{2} - \delta$ for some positive $\delta$.

\end{proof}

%%%%%%%%
\section{Holomorphic projection calculation}\label{sec:proofs}
\subsection{Lemmas}

We record a few integral formulas that we refer to in the proof of \Cref{thm:proj-computation}.

\begin{lemma}\label{thm:d1integral}
Let $a>0$, then we have the formulas
\begin{equation}\label{eq:d1integral}
 \int_0^{\infty} e^{- 4 \pi a y}\log {y} dy = - (4 \pi a)^{-1} (\gamma) + \log (4 \pi a))
\end{equation}
and 
\begin{equation}\label{eq:d2integral}
  \int_0^{\infty} y^{1/2} e^{-4 \pi  a y} dy = \frac{1}{16 \pi a^{3/2}}
  \end{equation}
and
\begin{equation}\label{eq:d3integral}
 \int_0^{\infty} e^{- 4\pi ay} y^{1/2} \log {y} dy = -(16\pi a^{3/2})^{-1} \left(-2 + \gamma + \log(16\pi a)\right).
\end{equation}
\end{lemma}
\begin{proof}
The first equation is a straightforward application of integration by parts and (6.6.2) of \cite{NIST:DLMF}. The second equation follows from $\Gamma(3/2) = \sqrt{\pi}/2$ and a $u$-substitution.

After a $u$-substitution, the integral in the third equation can be decomposed as
$$
    \int_0^{\infty} e^{- 4\pi ay} y^{1/2} \log {y} dy = (4 \pi a)^{- \frac{3}{2}} \left( \int_0^{\infty} e^{- y} y^{1/2} \log {y} dy - \log (4 \pi a) \int_0^{\infty} e^{-y} y^{1/2} dy \right) 
    $$ 
Both integrals are expressible via the $\Gamma$ function.  In particular,
$$
\int_0^{\infty} e^{- 4\pi ay} y^{1/2} \log {y} dy= (4 \pi a)^{- \frac{3}{2}} \left( \Gamma'(3/2) - \log (4 \pi a) \Gamma(3/2) \right)  
$$
Let $\psi$ be the digamma function. We rewrite the previous expression as
$$
\int_0^{\infty} e^{- 4\pi ay} y^{1/2} \log {y} dy
     = (4 \pi a)^{- \frac{3}{2}} \Gamma(3/2) \left(\psi(3/2) - \log (4 \pi a)  \right) 
     $$
Finally, we use (5.4.15) of \cite{NIST:DLMF} to evaluate the digamma value:
$$
\psi \left( \frac{3}{2} \right) = \gamma - 2 \log 2 + 2.
$$
\end{proof}
For $m,n>0$, set
\begin{equation}\label{eq:alpha-defn}
    \alpha_{n,m} \colonequals 4\pi (n+m^2) \int_0^{\infty}\alpha(4ny)e^{-4\pi (n+m^2)y}dy.
\end{equation}
\begin{lemma}\label{thm:alphaintegral}
If $n,m>0$, then
\[
    \alpha_{n,m} = \frac{1}{4 \pi m} \left( 2 \sqrt{n} \arctan\left( \frac{m}{\sqrt{n}}\right) - m \log \left(\frac{4n}{n+m^2}\right) \right).
\]
\end{lemma}
\begin{proof}
Applying Fubini's Theorem, we obtain
\begin{align*}
    \alpha_{n,m}     &= 4\pi N \int_0^{\infty}  \frac{\sqrt{4ny}}{4 \pi} \left(\int_0^{\infty} t^{-1/2} e^{- 4 \pi n y t} \log ( 1 + t) dt\right)  e^{- 4 \pi (n+m^2) y} dy  \\
    &= 2\sqrt{n}N \int_0^{\infty}   t^{-1/2} \log ( 1 + t) \left(\int_0^{\infty} \sqrt{y}  e^{- 4 \pi n y t}   e^{- 4 \pi (n+m^2) y} dy\right) dt 
\end{align*}
Let $N = n + m^2$. We use Lemma \ref{thm:d1integral} to evaluate the inner integral, and obtain
\begin{equation}\label{eq:alpha-integral}
\alpha_{n,m}= \frac{N}{8 \pi}  \int_0^{\infty}   t^{-1/2} \log ( 1 + t) ( t +(N/n))^{-3/2}  dt.
\end{equation}

It is convenient to set $a = N/n$ and decompose the integral as
\begin{align*}
    \int_0^{\infty} &t^{-1/2} (t+a)^{-3/2} \log (1+t) dt = \int_0^{\infty} t^{-2} (1 + (a/t))^{-3/2} \log(1+t) dt \\
    %&=  \textcolor{orange}{\int_0^{\infty} (1+au)^{-3/2} \log(1+(1/u)) du }\\
    &= \int_0^{\infty} (1+au)^{-3/2} \log (1+u) du 
    - \int_0^{\infty} (1 + au)^{-3/2} \log(u) du \\
    &=: I_2 - I_1
\end{align*}

The first integral can be evaluated using equation (6) on p. 544 of \cite{GradshteynRyzhik}, with $\mu = 1/2$.
\begin{align*}
    I_1 &=  a^{-3/2} \int_0^{\infty} ((1/a) + u)^{-3/2} \log(u) du \\
    &=  a^{-3/2} 2 a^{1/2} \left( \log(1/a) - \gamma-\psi\left(\frac{1}{2}\right) \right) \\
    &=  4 a^{-1} \log (2/\sqrt{a})
\end{align*}
where we have used 5.4.13 of \cite{NIST:DLMF} and again $\psi$ denotes the digamma function.

For the second integral, we use integration by parts. 
\begin{align*}
I_2 &= \left( \frac{-2}{a} \cdot \frac{ \log (1+u)}{(1+au)^{1/2}}     \right)\biggr\vert_0^{\infty} + \frac{2}{a} \int_0^{\infty} \frac{1}{1+u} \frac{1}{\sqrt{1+au}} du \\
%&= 0 + \frac{2}{a} \int_1^{\infty} \frac{1}{u} \frac{1}{\sqrt{ au + (1-a)}} du \\
%&=  \textcolor{orange}{\frac{2}{a^{3/2}} \int_1^{\infty} \frac{1}{u} \frac{1}{\sqrt{ u + ((1/a)-1)}} du} \\
%&= \textcolor{orange}{ \frac{2}{a^{3/2}} \int_0^{1} \frac{1}{t} \frac{1}{\sqrt{ (1/t) + ((1/a)-1)}} dt} \\
%&=  \frac{2}{a^{3/2}} \int_0^{1}  \frac{t^{-1/2}}{\sqrt{ 1 + t ((1/a)-1)}} dt \\
%&=  \frac{4}{a^{3/2}} \int_0^{1}  \frac{1}{\sqrt{ 1 -  v^2 (1 - (1/a))}} dv \\
&=  \frac{4 a^{1/2} }{a^{3/2} (a-1)^{1/2}}  \int_0^{1}  \frac{1}{\sqrt{ \frac{a}{a-1} -  v^2}} dv \\
%&=   \frac{4 a^{1/2} }{a^{3/2} (a-1)^{1/2}}  \arctan(\frac{u}{\sqrt{\frac{a}{a-1}}})|_0^1 \\
&=   \frac{4  }{a (a-1)^{1/2}}  \arctan \left(\sqrt{ a-1} \right) \\
\end{align*}
Plugging these expressions for $I_1$ and $I_2$ into \eqref{eq:alpha-integral} produces the result.
\end{proof}
%

%%%%%%%%%%%%%%%%%%%%%%%%%%%%

\begin{lemma}\label{thm:harmonic-int}
For $s \ge 0$, $N>0$, and $n > -N$, we have
$$
\int_0^{\infty} y^s \Gamma(1/2, 4 \pi ny) e^{-4 \pi Ny} dy =     \frac{\Gamma(3/2 + s) {}_2F_1  \left(1 + s, 3/2 + s; 2 + s, 
  -\frac{N}{n} \right) }{(1 + s)(4 \pi n)^{1+s}}
$$
\end{lemma}
\begin{proof}
Use (6.455) of \cite{GradshteynRyzhik} with $\mu = s+1$, $\nu=\frac{1}{2}$, $\alpha = 4 \pi n$, and $\beta = 4 \pi N$. Then apply 15.8.1 \cite{NIST:DLMF} with $a = 1$, $b = \frac{3}{2}+s$, $c = 2+s$, and $z = \frac{N}{N+n}$.
\end{proof}

Specializing to $s =0$ we obtain the following.
\begin{lemma}\label{thm:harmonic-int-0}
For $m,n,N>0$, where $N=m^2-n$, we have
$$
\int_0^{\infty} \Gamma(1/2, 4 \pi ny) e^{-4 \pi Ny} dy = \frac{1}{4 \sqrt{\pi}  (m + \sqrt{n})m}
$$
\end{lemma}
\begin{proof}
Evaluate the hypergeometric function in the previous lemma using 15.4.18 of \cite{NIST:DLMF} with $a=1$.
\end{proof}

%%%%%%

%%%%%
\subsection{Fourier calculation of projection of mixed sesquiharmonic forms}

\begin{theorem}\label{thm:proj-computation}
Let $F \in V_{1/2} (4N)$ have a Fourier expansion as in \eqref{eq:SesquiFourier}. Let $g = \sum_{m=1}^{\infty} \ell(m) q^{m^2} \in S_{\frac{3}{2}} (4N, \chi)$ be a cusp form with coefficients supported on square indices. We furthermore assume that $F$ and $g$ satisfy the properties that $F \cdot g \in \mathbb{S}_2(\Gamma_1(4N))$ and for any $h >0$, for all $\epsilon >0$ we have $b(m^2 - h) \ell(m) = O(m^{2 + \epsilon})$ as $m \to \infty$, where $b(m)$ is as in \eqref{eq:SesquiFourier}. 
Then we have
\[
    \pi_2^{hol} (F g) (z) = \sum_{h=1}^\infty a_h q^h,
\]
where
\begin{align*}
    &a_h = \left(\sum_{\substack{n+m^2 = h \\ n, m > 0}} a(n) \ell(m) \alpha_{n, m} \right) + \sum_{\substack{n + m^2 = h \\ m>0 \\ n \neq 0}} c(n) \ell(m) \\
    &+  \delta_{\square}(h) \ell(\sqrt{h}) \left(\left(d_0 - d_1(\log(4\pi h) + \gamma)\right) + \frac{1}{4 \sqrt{h}}\left( d_2 + d_3(2- \log(16 \pi h) - \gamma \right) \right) \\
    &+  \sqrt{\pi} \sum_{h=1}^{\infty} h q^h  \left( \sum_{\substack{ n + m^2 = h \\ n \ll \infty \\ m > 0}} \!\!\!\!\! \frac{ b(n) \ell(m)}{ (m + \sqrt{|n|})m} \right).
\end{align*}
Here $\alpha_{n,m}$ %and $\beta_{n,m}$
is as defined in \eqref{eq:alpha-defn}
and $\delta_{\square}(N)$ is defined to be $1$ if $N$ is a perfect square and $0$ otherwise.
\end{theorem}
\begin{proof}
Write $F(z) = F_0(z) + F_1(z) + F_2(z) + F_3(z) $ for the four sums in \eqref{eq:SesquiFourier}. We use Theorem \ref{thm:proj-formula} to compute the holomorphic projection of $F_i g$ for each $i$. Recall that from \eqref{eq:SesquiFourier}, the constant term $F_0$ is given by
    $$
    F_0(z) = d(0) + d(1) \log (y) + d(2) y^{\frac{1}{2}} + d(3)  y^{\frac{1}{2}} \log (y).
    $$

   When applying Theorem \ref{thm:proj-formula} to compute $\pi_2^{hol}(F_0 \cdot g)$, interchanging the integral and limit is justified by the exponential decay of the integrand at infinity and the slow growth at 0. Combining \Cref{thm:proj-formula} and \Cref{thm:d1integral} we obtain
    \begin{singnumalign}\label{eq:proj-constant}
        \pi_2^{hol}(F_0 \cdot g  )(z)=& \sum_{m=1}^{\infty} \ell(m) \biggr(\left(d(0) - d(1)(\log(4\pi m^2) + \gamma)\right) \\
        &\hspace{20mm}+ \frac{1}{4 m}\left( d(2) + d(3)(2- \log(16 \pi m^2) - \gamma \right) \biggr) q^{m^2}.
    \end{singnumalign}
   The holomorphic part $F_1(\tau)g(\tau)$ is fixed by the holomorphic projection operator, i.e.
    \begin{equation}\label{eq:proj-holo}
        \pi_2^{hol}(F_1 \cdot g)(z) = \sum_{h=1}^{\infty} \sum_{\substack{n + m^2 = h \\ n \gg  -\infty \\ m > 0}} c(n) \ell(m) q^h.
    \end{equation}
   Next we turn our attention to the harmonic part 
    \begin{align*}
        F_2(z)g(z) &= \left( \sum_{n \ll \infty} b(n) \beta(-4ny)q^n \right) \left( \sum_{m > 0} \ell(m) q^{m^2} \right) \\
      &= \sum_{h =1}^{\infty} \left( \sum_{\substack{n+m^2 = h \\ n \ll \infty \\ m > 0 }} b(n) \ell(m) \beta(-4n y) \right) q^h.
    \end{align*}
    By Theorem \ref{thm:proj-formula} we have
    \begin{singnumalign}\label{eq:proj-harm}
        \pi_2^{hol}(F_2 \cdot g  )(z)&= \sum_{h=1}^{\infty} 4\pi h \lim_{s \to 0 } \left( \int_0^{\infty} y^s \!\!\!\!\!\! \sum_{\substack{ n + m^2 = h \\ n \ll \infty \\ m > 0}} \!\! b(n) \ell(m) \left(  \beta(-4 ny)e^{-4\pi(n+m^2)y}\right) dy \right) q^h \\
    \end{singnumalign}
           
    Let $0 < \epsilon < s$. Using our assumption on $b(m^2 - h) \ell(m)$, we find that there exists a constant $C_{\epsilon}$ such that for all $y,m > 0$, we have
    $$
    |b(m^2 - h) \ell(m) e^{- 2 \pi m^2 y}| \le C_{\epsilon} y^{- \epsilon - \frac{1}{2}} .
    $$

    Therefore we can bound the integrand in \eqref{eq:proj-harm} by a constant multiple of
    \begin{equation}\label{eq:bound}
    y^s y^{- \epsilon - \frac{1}{2}} \sum_{m=1}^{\infty}  e^{- 2 \pi m^2 y} =   y^{s- \epsilon - \frac{1}{2}} (\theta(iy) - 1).
    \end{equation}
    This function is convergent for all $y \in (0,\infty)$. It decays exponentially as $y \to \infty$. Using the transformation law for the Jacobi theta function, $\theta(iy)-1 = O(y^{-1/2})$ as $ y \to 0$. It follows that the function in \eqref{eq:bound} is integrable on $(0,\infty)$, and therefore we may swap the inner sum with the integral in \eqref{eq:proj-harm}.            
    \begin{align*}
        \pi_2^{hol}&(F_2 \cdot g (z)) = \sum_{h=1}^{\infty} 4\pi h \lim_{s \to 0} \left( \sum_{\substack{ n + m^2 = h \\ n \ll \infty \\ m > 0}} \!\!\!\!\! b(n) \ell(m)  \int_0^{\infty} y^s \left(  \beta( - 4 ny)e^{-4\pi(n+m^2)y}\right) dy \right) q^h \\
        &= \sum_{h=1}^{\infty} 4\pi h \lim_{s \to 0} \left( \sum_{\substack{ n + m^2 = h \\ n \ll \infty \\ m > 0}} \!\!\!\!\! b(n) \ell(m)  \int_0^{\infty} y^s \frac{1}{\sqrt{\pi}} \Gamma \left( \frac{1}{2}, - 4 \pi n y \right) e^{-4\pi(n+m^2)y} dy \right) q^h \\
        \end{align*}
        
      A calculation using Proposition \ref{thm:ShiftedConvBound} shows that the inner sum is uniformly convergent in $s$ near 0, so we can swap the limit and sum and obtain, using Lemma \ref{thm:harmonic-int-0},
\begin{singnumalign}\label{eq:proj-harm-interchanged}
\pi_2^{hol}&(F_2 \cdot g (z)) = \sum_{h=1}^{\infty} 4\pi h \left( \sum_{\substack{ n + m^2 = h \\ n \ll \infty \\ m > 0}} \!\!\!\!\! b(n) \ell(m)  \int_0^{\infty} \frac{1}{\sqrt{\pi}} \Gamma \left( \frac{1}{2}, - 4 \pi n y \right) e^{-4\pi(n+m^2)y} dy \right) q^h \\
&= \sqrt{\pi} \sum_{h=1}^{\infty} h q^h  \left( \sum_{\substack{ n + m^2 = h \\ n \ll \infty \\ m > 0}} \!\!\!\!\! \frac{ b(n) \ell(m)}{ (m + \sqrt{|n|})m} \right) \\
\end{singnumalign}
      
    And finally, we have the sesquiharmonic piece
    \begin{align*}
        F_3(z)g(z) = \sum_{h \geq 1} \left( \sum_{\substack{n+m^2 = h \\ n, m > 0}} a(n) \ell(m) \alpha(4n y) \right) q^h.
    \end{align*}
    We interchange the limit and integral in Theorem \ref{thm:proj-formula}, and using \Cref{thm:alphaintegral}, we obtain
    \begin{singnumalign}\label{eq:proj-nonharm}
        \pi_2^{hol}(F_3 \cdot g )(z)&= \sum_{h=1}^{\infty} 4\pi h \left( \sum_{\substack{n+m^2 = h \\ n, m > 0}} a(n)\ell(m) \int_0^{\infty} \alpha(4ny)e^{-4\pi (n+m^2)y}dy \right) q^h \\
        &= \sum_{h=1}^\infty 
        \left(\sum_{\substack{n+m^2 = h \\ n, m > 0}} a(n) \ell(m) \alpha_{n,m} \right)q^h
    \end{singnumalign}
    Combining \eqref{eq:proj-constant}, \eqref{eq:proj-holo}, \eqref{eq:proj-harm-interchanged}, and \eqref{eq:proj-nonharm} completes the proof.
\end{proof}
%

%%%%
\subsection{Proof of Theorem \ref{thm:AAS-projection}}
Let $F = Z$ and $g = \theta_{\chi}$ in Theorem \ref{thm:proj-computation}.

%%%%%%%%
\section{Example}\label{sec:numerics}

Computing out to 50000 terms in the harmonic part of $r_{\chi_4}(k)$ gives us the approximation
\begin{equation}\label{eq:approximated-projection}
    \pi_2^{hol}\left((Z(\tau) \theta_{\chi_4}(\tau)\right)) \approx (0.028599) f_1 + (0.0000017) f_2 + (0.0579284) f_3,
\end{equation}

where the $f_i \sum_{n=1}^{\infty} c_i(n) q^n$ are the holomorphic cusp forms spanning $S_2(64)$ as defined in the introduction. We note that the $n^{th}$ coefficient of both the right and left hand sides can be seen directly to be zero if $n \equiv 0,3 \pmod{4}$.  Computing the left hand side of \eqref{eq:approximated-projection} using Theorem \ref{thm:AAS-projection} requires a truncation of the harmonic sums. In the table below, we compute several values of $r_{\chi_4}(k)$---as defined in \Cref{thm:AAS-projection} and approximating the harmonic contribution by $T_0(N)$ truncated to 10000 terms---and compare to the right hand side of \eqref{eq:approximated-projection}.  These computations were done using SageMath \cite{Sagemath}.
\begin{figure}[h!]
    \begin{tabular}{|c|c|c|c||c|c|c|c|}
        \hline
        $k$ & Numerical & Expected & Abs. Error & $k$ & Numerical & Expected & Abs. Error \\
        \hline
        1 & 0.0289 & 0.0286 & 0.0003 & 50 & -0.0573 & -0.0579 & 0.0006 \\
        \hline
        2 & 0.058 & 0.0579 & 0.0001 & 53 & -0.4032 & -0.4004 & 0.0028 \\
        \hline
        5 & 0.0577 & 0.0572 & 0.0005 & 54 & -0.001 & 0 & 0.001 \\
        \hline
        6 & -0.0001 & 0 & 0.0001 & 57 & 0 & 0 & 0 \\
        \hline
        9 & -0.0869 & -0.0858 & 0.0011 & 58 & -0.5769 & -0.5793 & 0.0024 \\
        \hline
        10 & -0.1163 & -0.1159 & 0.0004 & 61 & 0.2866 & 0.286 & 0.0006 \\
        \hline
        13 & -0.1738 & -0.1716 & 0.0022 & 62 & -0.0011 & 0 & 0.0011 \\
        \hline
        14 & 0.0001 & 0 & 0.0001 & 65 & -0.3501 & -0.3432 & 0.0069 \\
        \hline
        17 & 0.0576 & 0.0572 & 0.0004 & 66 & -0.002 & 0 & 0.002 \\
        \hline
        18 & -0.174 & -0.1738 & 0.0002 & 69 & -0.0006 & 0 & 0.0006 \\
        \hline
        21 & -0.0002 & 0 & 0.0002 & 70 & -0.0005 & 0 & 0.0005 \\
        \hline
        22 & -0.0005 & 0 & 0.0005 & 73 & -0.1737 & -0.1716 & 0.0021\\
        \hline
        25 & -0.03 & -0.0286 & 0.0014 & 74 & -0.115 & -0.1159 & 0.0009 \\
        \hline
        26 & 0.3463 & 0.3476 & 0.0013 & 77 & 0 & 0 & 0 \\
        \hline
        29 & 0.2898 & 0.286 & 0.0038 & 78 & 0.0012 & 0 & 0.0012 \\
        \hline
        30 & 0.0001 & 0 & 0.0001 & 81 & 0.2559 & 0.2574 & 0.0015 \\
        \hline
        33 & -0.0001 & 0 & 0.0001 & 82 & 0.5786 & 0.5793 & 0.0007 \\
        \hline
        34 & 0.1149 & 0.1159 & 0.001 & 85 & 0.1167 & 0.1144 & 0.0023 \\
        \hline
        37 & 0.0556 & 0.0572 & 0.0016 & 86 & -0.0004 & 0 & 0.0004 \\
        \hline
        38 & 0.006 & 0 & 0.006 & 89 & 0.2891 & 0.286 & 0.0031 \\
        \hline
        41 & 0.2894 & 0.286 & 0.0034 & 90 & 0.344 & 0.3476 & 0.0036 \\
        \hline
        42 & -0.0022 & 0 & 0.0022 & 93 & -0.0024 & 0 & 0.0024 \\
        \hline
        45 & -0.1746 & -0.1716 & 0.0058 & 94 & -0.0009 & 0 & 0.0009 \\
        \hline
        46 & -0.001 & 0 & 0.001 & 97 & 0.5199 & 0.5148 & 0.0051 \\
        \hline
        49 & -0.2038 & -0.2002 & 0.0036 & 98 & -0.4276 & -0.4055 & 0.0221 \\
        \hline
    \end{tabular}
    \caption{The numerical column gives approximations to the nearest $10,000^{th}$ for $r_{\chi_4}(k)$ using the holomorphic projection computed in \Cref{thm:proj-computation}, the expected column gives $r_{\chi_4}(k)$ computed as a linear combination of Fourier coefficients of modular forms as in the right-hand side of \Cref{eq:approximated-projection}.}
\end{figure}

We note a few arithmetic properties of the $r_{\chi_4}(k)$ that are consequences of Theorem \ref{thm:AAS-projection}. For example, $f_1,f_2$,and $f_3$ have integer coefficients, and the Fourier expansions of $f_1$, and $f_2$ are supported on indices which are $1 \pmod{4}$, while the Fourier expansion of $f_3$ is supported on indices which are 2 modulo 4. Furthermore, $f_1$ is the twist of $f_2$ by the Kronecker character $\chi_8$ (it is easy to verify this by checking that the minimal Weierstrass equation of the elliptic curve associated to $f_1$, which has Cremona label 64a3, is the discriminant 8 quadratic twist of the elliptic curve associated to $f_2$, which has Cremona label 32a4). Consequently:
$$
r_{\chi_4} (k) \in 
\begin{cases}
r_{\chi_4} (2) \mathbb{Z} & k \equiv 2 \pmod{4} \\
r_{\chi_4} (1) \mathbb{Z} & k \equiv 1 \pmod{8} \\
\frac{1}{2} r_{\chi_4}(5) \mathbb{Z} & k \equiv 5 \pmod{8}. 
\end{cases}
$$
 Numerically we seem to have $r_{\chi_4}(1) = \frac{1}{2} r_{\chi_4}(5)$, which would of course allow us to combine the second and third cases. 
 
 Additionally, the Hecke relations for the $f_i$ produce corresponding relations for the $r_{\chi_4}(k)$. For odd primes $p$, we have
 $$
  T_p f_i := \sum_{n=1}^{\infty} (c_i(np) + p c_i(n/p) ) q^n = \begin{cases}
      c_i(p) f_i & i = 1,2 \\
      c_2(p) f_3 & i =3.
  \end{cases}
 $$
 From this, we see that for $(n,p)=1$, we have
$$
r_{\chi_4}(2pn) = c_2(p) r_{\chi_4}(2n)
$$
and if $p \equiv 1 \pmod{8}$ (i.e. $c_1(p) = c_2(p)$), 
then
$$
r_{\chi_4}(pn) = c_1(p) r_{\chi_4}(n)
$$
for all $n$ co-prime to $2p$. From the fact that $c_1(p) = c_2(p) = 0$ for $p \equiv 3 \pmod{4}$, we have $r_{\chi_4}(k)=0$ if $p||k$ and $k$ is odd, where $p$ is any prime congruent to 3 modulo 4.

\section*{Acknowledgements}
The first author was partially supported by Simons Foundation Collaboration Grant \#953473 and National Science Foundation Grant DMS-2401356. The authors thank the anonymous referees for their thorough and helpful feedback.

\bibliographystyle{halpha-abbrv}
\bibliography{references.bib}     
\end{document}